\documentclass[11pt,a4paper]{amsart}
\usepackage[foot]{amsaddr}
\usepackage{texmac}
\usepackage{enumerate}
\usepackage{mathrsfs}
\usepackage[all]{xy}
\xyoption{web}
\operators{Stab,Prim,Inn,Out,ss,B,G,R,K,M,Inf,Def,Imprim}

\bbsymbols{b}{C,F,Z,Q}
\calsymbols{c}{H,U,X,Y,P,M,D,F,G,K,I,Q,S,C}
\let\le\leqslant
\let\ge\geqslant

\def\dsum_#1_#2{\sum_{{#1}\atop {#2}}}
\def\rar{\rightarrow}
\begin{document}

\title[Brauer relations in positive characteristic semisimplified.]
{Brauer relations for finite groups in the ring of semisimplified modular representations}

\author{Matthew Spencer}
\email{ matthew.james.spencer91@gmail.com}
\address{Mathematics Institute, Zeeman Building, University of Warwick,
	Coventry CV4 7AL, UK}
\llap{.\hskip 10cm} \vskip -0.8cm

\maketitle

\begin{abstract}
Let $G$ be a finite group and $p$ be a prime. We study the kernel of the map, between the Burnside ring of $G$ and the Grothendieck ring of $\mathbb{F}_p[G]$-modules, taking a $G$-set to its associated permutation module. We are able, for all finite groups, to classify the primitive quotient of the kernel; that is for each $G$, the kernel modulo elements coming from the kernel for proper subquotients of $G$.  We are able to identify exactly which groups have non-trivial primitive quotient and we give generators for the primitive quotient in the soluble case.
\end{abstract}

\section{Introduction}
In this paper we will describe, for a prime $p$ and a finite group $G$, the kernel of the map $m_{\mathbb{F}_p,ss}(G)$ from the Burnside ring $b(G)$ to the Grothendieck ring of $\mathbb{F}_p[G]$ modules $G_0(\mathbb{F}_p[G])$ (equivalently the ring of Brauer characters), 
which takes the isomorphism class of a finite $G$-set $X$ to the class of $\mathbb{F}_p[X]$ in $G_0(\mathbb{F}_p[G])$. The $\mathbb{F}_p[G]$-module $\mathbb{F}_p[X]$ has an $\mathbb{F}_p$-basis indexed by elements of $X$, $G$ acts by permuting this basis in the obvious way.
The class of $\mathbb{F}_p[X]$ in $G_0(\mathbb{F}_p[G])$ is equal to that of its semisimplification and is by definition determined by its composition factors. Note that composition factors do not in general determine isomorphism class, so two isomorphism classes of $G$-sets may have the same image in $G_0(\mathbb{F}_p[G])$ even 
if their associated permutation modules are not isomorphic. 
\par
\begin{definition}

Let $G$ be a finite group. The additive group of the \emph{Burnside ring}   $b(G)$ is the free abelian group on isomorphism classes of finite $G$ sets modulo the relations $[X]+[Y]-[X\coprod Y]$, where $[X]$ denotes the isomorphism class of $X$. Furthermore,  $b(G)$ is equipped with a ring structure defined by $[X]\cdot [Y]=[X\times Y]$.
\end{definition}
\begin{remark}
The Burnside ring is isomorphic to the free abelian group on isomorphism classes of transitive $G$ sets. Moreover, by the orbit-stabiliser theorem each transitive $G$-set is isomorphic to  $G/H$, where $H\le G$ is a stabiliser of an element of the $G$-set, and is determined uniquely up to $G$-conjugacy.
Thus we may identify $b(G)$, as an abelian group, with the free abelian group on conjugacy classes of subgroups of $G$.
\end{remark}
\begin{definition}
Let $G$ be a finite group and $R$ be a commutative ring, the \emph{Grothendieck Group} of $R[G]$ modules $G_0(R[G])$ is defined to be the free abelian group on isomorphism classes of
$R[G]$-modules modulo the relations $[A]-[C]+[B]$ for every short exact sequence of $R[G]$-modules:
$$0\rar A \rar C \rar B \rar 0.$$
Furthermore, there is a ring structure with multiplication given by $[A]\cdot [B]=[A\otimes_R B]$.
\end{definition} 
 Consider the map:
 \begin{align*}
 	m_{\mathbb{F}_p,ss}(G): b(G) &\rar G_0(\mathbb{F}_p[G])
 	\\ [H] &\mapsto [\Ind_{G/H}(1)],
 	\end{align*}
 	where $\Ind_{G/H}(-)$ denotes the induction map from $H$ to $G$ and $1$ denotes the trivial $\mathbb{F}_p[H]$-module. Elements of the kernel $K_{\mathbb{F}_p,ss}(G)$ will be called  \emph{Brauer Relations} for $G$ over $\mathbb{F}_p$ semisimplified or relations over $\mathbb{F}_p,ss$ for short.\\ Following
\cite{BD} and \cite{GFI}  a relation for $G$ is said to be imprimitive if it is a linear combination of relations which are ``inflated'' from proper quotients of $G$ or ``induced'' from proper subgroups. We will then describe, for a finite group $G$ and prime $p$, 
the structure of the kernel modulo imprimitive relations 
 as an abelian group and describe explicit generators in the case $G$ is soluble. More precisely our main aim is the following theorem:
\begin{theorem}
Let $p$ be a prime, and $G$ be a finite group of order divisible by $p$,  all elements of $\ker(m_{\mathbb{F}_p,ss}(G))$ are linear combinations of relations ``induced'' or ``inflated'' from $\ker(m_{\mathbb{F}_p,ss}(H))$ from subquotients  $H$ of the following forms:
\begin{enumerate}
	\item a cyclic group $C_p$ of order $p$,
	\item non-cyclic $q$-quasi-elementary groups with order coprime to $p$,
	\item $(C_l\rtimes C_{q^r})\times (C_l\rtimes C_{q^s})$ for primes $q,l$ with $l\ne p$ and the action faithful,
	\item an extension $1 \rar S^d \rar E \rar H \rar 1$ with $S$ simple, $d$ a positive integer, $H=C\rtimes Q$ a quasi-elementary group whose cyclic part $C$ is of order coprime to $p$, and $S^d$ is a unique minimal normal subgroup of $E$.
\end{enumerate}
In cases (1),(3) and the soluble case of (4) the construction of generating elements of  $\ker(m_{\mathbb{F}_p,ss}(H))$ is given explicitly.
\end{theorem}
\par

\begin{acknowledgements}

The author was supported by an EPSRC Doctoral Grant and would like to thank them for their financial support. The author would like to thank the reviewer for their helpful comments which helped improve the quality and clarity of this manuscript. The author would also like to thank Alex Bartel for helpful remarks on an earlier version and for discussions during the development of this paper. The author is also indebted to Gareth Tracey and George Callender for their careful comments on earlier drafts. Finally the author would like to thank Alex Torzewski for helpful comments and advice, and discussions while preparing this paper.

\end{acknowledgements}

\section{Primordial and coprimordial groups for $\Im(m_{\mathbb{F}_{p},ss})$}
We will now introduce Mackey and Green functors with inflation.
The following definitions can also be found in \cite[Section 2]{GFI} following  Webb \cite[Section 1]{Webb2}. We will assume that $R$ is a domain.
\begin{definition}
A \emph{global Mackey functor with inflation} (MFI)
over $R$ is a collection $\cF$ of the following data.
\begin{itemize}
\item For every finite group $G$, $\cF(G)$ is an $R$-module;
\item for every injection $\alpha\colon H\hookrightarrow G$ of finite groups,
$\cF_*(\alpha)\colon \cF(H)\rightarrow \cF(G)$ is a covariant $R$-module homomorphism
(which we think of as induction);
\item for every homomorphism $\epsilon\colon H\rar G$ of finite groups,
$\cF^{*}(\epsilon)\colon \cF(G)\rar \cF(H)$ is a contravariant $R$-module
homomorphism (which we think of as restriction when $\epsilon$ is a injection,
and as inflation when $\epsilon$ is an surjection);
\end{itemize}
satisfying the following conditions.
\begin{enumerate}
\item[(MFI 1)] Transitivity of induction:
for all group injections
$U\stackrel{\beta}{\hookrightarrow}H\stackrel{\alpha}{\hookrightarrow}G$, we
have $\cF_*(\alpha\beta) = \cF_*(\alpha)\cF_*(\beta)$.
\item[(MFI 2)] Transitivity of restriction/inflation:
for all group homomorphisms
$U\stackrel{\beta}{\rar}H\stackrel{\alpha}{\rar}G$, we have
$\cF^*(\alpha\beta) = \cF^*(\beta)\cF^*(\alpha)$.
\item[(MFI 3)] For all inner automorphisms $\alpha\colon G\rar G$, we have
$\cF^*(\alpha)=\cF_*(\alpha)=~1$.
\item[(MFI 4)] For all automorphisms $\alpha$, we have
$\cF_*(\alpha)=\cF^*(\alpha^{-1})$.
\item[(MFI 5)] The Mackey condition: for all pairs of injections
$\alpha\colon H\hookrightarrow G$ and $\beta\colon K\hookrightarrow G$, 
$$
\cF^*(\beta)\cF_*(\alpha)=\sum_{g\in \alpha(H)\backslash G/\beta(K)}\cF_*(\phi_g)\cF^*(\psi_g),
$$
where $\phi_g$ is the composition
$$
\phi_g\colon \beta(K)^g\cap \alpha(H)\stackrel{c_g}{\rar}
\beta(K)\cap {}^g\alpha(H)\hookrightarrow \beta(K)\stackrel{\beta^{-1}}{\rar}K,
$$
$c_g$ denoting conjugation by $g$,
and $\psi_g$ is the composition
$$
\psi_g\colon \alpha(H)\cap \beta(K)^g\hookrightarrow \alpha(H)
\stackrel{\alpha^{-1}}{\rar}H.
$$
\item[(MFI 6)] Commutativity of induction and inflation: whenever there is a commutative
diagram
$$
\xymatrix{
H\ar[d]_{\epsilon}\ar[r]^{\alpha}& G\ar[d]^{\delta}\\
\bar{H}\ar[r]^{\beta}& \bar{G},
}
$$
where $\epsilon,\delta$ are surjections, and $\alpha,\beta$ are injections,
we have $\cF^{*}(\delta)\cF_*(\beta)=\cF_*(\alpha)\cF^{*}(\epsilon)$.
\end{enumerate}
\end{definition}
We will henceforth write $\Res_{G/H}(-)$ and $\Inf_{G/N}(-)$ for $\cF^*(\alpha)(-)$  when $\alpha$ is an injection or surjection respectively. We will write $\Ind_{G/H}(-)$ for $\cF_*(\beta)(-)$. 
If a Mackey functor takes Groups to $R$-algebras, inflation and restriction are $R$-algebra homomorphisms, and if it satisfies a form of Frobenius reciprocity, then we call it a Green functor which we will now define precisely.
\begin{definition}
A \emph{Green functor with inflation} (GFI) over $R$ is an MFI $\cF$ over $R$,
satisfying the following additional conditions.
\begin{enumerate}
\item[(GFI 1)] For every finite group $G$, $\cF(G)$ is an $R$-algebra.

\item[(GFI 2)] For every homomorphism $\alpha\colon H\rar G$ of finite groups,
$\cF^*(\alpha)$ is a homomorphism of $R$-algebras.

\item[(GFI 3)] Frobenius reciprocity: for every injection $\alpha\colon H\hookrightarrow G$
and for all $x\in \cF(H)$, $y\in \cF(G)$, we have
\beq
\Ind_{G/H}(x)\cdot y & = &
\Ind_{G/H}(x\cdot \Res_{G/H}(y)),\\
y\cdot\Ind_{G/H}(x) & = &
\Ind_{G/H}(\Res_{G/H}(y)\cdot x).
\eeq
\end{enumerate}
\end{definition}
\begin{definition}
A \emph{morphism} from an MFI (respectively GFI) $\cF$ to an MFI (respectively GFI)
$\cG$ is a collection $r(-)$ of $R$-module (respectively $R$-algebra) homomorphisms
$r(G): \cF(G)\rar \cG(G)$ for each finite group $G$, commuting in the obvious way
with $\cF_*,\cF^*,\cG_*,$ and $\cG^*$.
\end{definition}

\begin{notation}\label{not:indres}
Let $\cF$ be an MFI, and let $\cX$ be a class of groups
closed under isomorphisms. For every finite group $G$, we
define the following $R$-submodules of~$\cF(G)$:
\begin{eqnarray*}
\cI_{\cF,\cX}(G) & = & \sum_{H\leq G, \;H\in \cX}\Ind_{G/H}\cF(H),\\
\cI_{\cF}(G) & = & \sum_{H\lneq G}\Ind_{G/H}\cF(H),\\
\cK_{\cF,\cX}(G) & = & \bigcap_{H\leq G, \;H\in \cX} \ker(\Res_{G/H} \cF(G)),\\
\cK_{\cF}(G) & = & \bigcap_{H\lneq G} \ker(\Res_{G/H} \cF(G)).
\end{eqnarray*}
\end{notation}
Following Th\'evenaz \cite{Thevenaz} and Boltje \cite{Boltje-GTCI} we make the following definition.
\begin{definition}\label{def:primordial}
Let $\cF$ be an MFI and let $G$ be a finite group.
\\
 We say that $G$ is
\emph{primordial} for $\cF$ if either $G$ is trivial, or
$\cF(G)\ne \cI_{\cF}(G)$. We denote the class
of all primordial groups for $\cF$ by $\cP(\cF)$.
\\
We say that $G$ is \emph{coprimordial} for $\cF$ if either $G$ is trivial, or
$\cK_{\cF}(G)\ne 0$. We denote the class of all coprimordial groups for $\cF$
by $\cC(\cF)$.
\end{definition}

\par
The functors taking a finite group $G$ to $b(G)$ and to $G_0(\mathbb{F}_p[G])$ for any prime $p$ are both GFIs over $\mathbb{Z}$, furthermore the map $m_{\mathbb{F}_p,ss}$ defined by:
\begin{align*}m_{\mathbb{F}_p,ss}(G) : b(G) &\longrightarrow G_0(\mathbb{F}_p[G])\\
[H] &\longmapsto \Ind_{G/H}(1)\end{align*}
is a morphism of GFIs over $\mathbb{Z}$. By  \cite[Lemma 2.8]{GFI} the kernel of this map $\K_{\F_p,ss}(-)$ is an ideal of $b(-)$.  As in the introduction, we will refer to elements of $\K_{\F_p,ss}(G)$ as \emph{Brauer relations} for $G$ over $\mathbb{F}_p$ semisimplified, or over $\mathbb{F}_p,ss$ as a shorthand. We exploit the machinery of \cite{GFI} to classify elements of the kernel. Note that this map of Green functors satisfy the assumptions of \cite{GFI} made in  Notation 4.3 with $D(G)=1$ and in Assumption 4.8 are satisfied so the results of \cite{GFI} hold in this case. 

\begin{notation}
	Let $G$ be a finite group.
	\\
	Let
	 $$\Imprim(G)=\{ \sum_{H<G}\Ind_{G/H}(\K_{\F_p,ss}(H))+\sum_{1 \ne N\triangleleft G} \Inf_{G/N}(\K_{\F_p,ss}(G/N))\},$$ and let $$\Prim_{\K_{\F_p,ss}}(G)=\K_{\F_p,ss}(G)/\Imprim(G).$$
\end{notation}

The following theorem,  well known in the literature (see for instance \cite{Benson}), is the prototypical example of an induction theorem. 
This will be the main tool we require to describe $K(-)$.\begin{theorem}\label{thm:Artin}[Artin's Induction Theorem]
	Let $G$ be a finite group, $p$ be a prime, and $G_0(\mathbb{F}_p[G])$ be the Grothendieck ring on $\mathbb{F}_p[G]$-modules then:
	$$ G_0(\mathbb{F}_p[G])\otimes \mathbb{Q}=\sum_{C\in S}\Ind_{G/C}( G_0(\mathbb{F}_p[C])\otimes \mathbb{Q})$$
	where S is a set of conjugacy class representatives of cyclic subgroups of order coprime to $p$. Furthermore, we have;
	$$1_{G_0(\mathbb{F}_p[G])}=\sum_{C\in S}a_c\Ind_{G/C}(1_{G_0(\mathbb{F}_p[C])})$$
	where $a_c\in \mathbb{Q}$.
\end{theorem}

We will make use of the following straightforward result, which follows from the work of Yoshida \cite{Yoshida}.
\begin{proposition}\label{prop:primordial}
	Let $\cF$ be a GFI over a Euclidean domain $R$, and assume that $\cF(G)$ is $R$-torsion free for
	all finite groups $G$. Let $Q$ denote the field of fractions of $R$. Then:
	\begin{enumerate}
		\item if  $Q$ has characteristic 0, then $\cP(\cF\otimes Q) = \cC(\cF\otimes Q)$;
		\item for any prime ideal $\mathfrak{p}$ of $R$, we have
		$\cC(\cF)=\cC(\cF\otimes R_{\mathfrak{p}})$, and in particular $\cC(\cF)=\cC(\cF\otimes Q)$;
		\item we have $\cP(\cF) \subseteq \cup_{\mathfrak{p}} \cP(\cF\otimes R_{\mathfrak{p}})$;
		\item if $Q$ has characteristic 0, then for any prime ideal $\mathfrak{p}$ of $R$,
		we have $\cP(\cF\otimes R_{\mathfrak{p}})\subseteq  \{ H : O^{p}(H)\in \cC(\cF)\}$,
		where $(p)=\mathfrak{p}\cap\bZ$.
	\end{enumerate}
\end{proposition}
\begin{proof}
	\mbox{}
	
\begin{enumerate}
		\item
		If $Q$ has characteristic 0, then by
		\cite[Theorem 7.1]{Dress} (see also Theorems 2 and 4 of \cite{Dress2}), we have a  decomposition of $Q$-modules
		$\cF(G)\otimes Q=\cK_{\cF\otimes Q}(G)\oplus \cI_{\cF\otimes Q}(G)$. The result follows.
		\item Since $\cF(G)$ is $R$-torsion free for all finite groups $G$,
		$\cF(G)$ naturally injects into $\cF(G)\otimes R_{\mathfrak{p}}$ and generates $\cF(G)\otimes R_{\mathfrak{p}}$
		over $R_{\mathfrak{p}}$ and similarly 
		over $Q$.
		Moreover, this inclusion is functorial with respect to restriction. It follows
		that we have a natural isomorphism $\cK_{\cF\otimes R{\mathfrak{p}}}(G) =  \cK_{\cF}(G)\otimes R_{\mathfrak{p}}$, and in particular one
		of these kernels is non-trivial if and only if both are, as claimed.
		\item Suppose that $G \notin \cup_{\mathfrak{p}} \cP(\cF\otimes R_{\mathfrak{p}})$ then in particular $1_{\cF \otimes R_{\mathfrak{p}(G)}}\in \cI_{\cF\otimes R_{\mathfrak{p}}}(G)$ for all $\mathfrak{p}$. Since $R$ is Euclidean it follows that $1_{\cF(G)} \in \cI_\cF(G)$ and as $\cI_\cF (G)$ is an ideal in $\cF(G)$ they coincide so $G \notin \cP(\cF)$.
		\item
		Let $G$ be a finite group, and let $p\in \mathbb{Z}$ be such that $\mathfrak{p} \mid p$. Let $\cH_p(\cC(\cF))=\{H: O^{p}(H)\in \cC(\cF)\}$. Let $(\#G)_{p'}$ denote the maximal divisor of $\#G$ which is coprime to $p$.
		Since $(\#G)_{p'}$ is invertible in $R_{\mathfrak{p}}$, \cite[Theorem 4.1]{Yoshida} applied with $\mathscr{X}=\{H\leq G: H \in \cC(\cF)\}$ implies that
		$\cF(G) \otimes R_{\mathfrak{p}}= \cI_{\cF\otimes R_{\mathfrak{p}},\cH_p(\cC(\cF))}(G)+
		\cK_{\cF\otimes R_{\mathfrak{p}},\cC(\cF)}(G)$. By part (2)
		of the present lemma, we have $\cC(\cF)=\cC(\cF\otimes R_{\mathfrak{p}})$. It then follows that $\cK_{\cF_{\mathfrak{p}},\cC(\cF)}(G)=
		\cK_{\cF_{\mathfrak{p}},\cC(\cF_{\mathfrak{p}})}(G)=0$ by definition \ref{def:primordial},
		and therefore that  $\cF_{\mathfrak{p}}(G)= \cI_{\cF_{\mathfrak{p}},\cH_p(\cC(\cF))}(G)$.
		So $\cP(\cF\otimes R_{\mathfrak{p}})\subseteq \cH_p(\cC(\cF))$, as claimed.
	\end{enumerate}
\end{proof}

\begin{lemma}\label{lem:copriss}
	The coprimordial groups for $\Im(m_{\mathbb{F}_p , ss}) \subseteq G_0(\mathbb{F}_p[-])$ are precisely the cyclic groups of order coprime to $p$.
\end{lemma}
\begin{proof}
	
	Theorem \ref{thm:Artin} combined with Proposition \ref{prop:primordial} show that $\cC(G_0(\mathbb{F}_p[-]))=\cP(G_0(\mathbb{F}_p[-]) \otimes \mathbb{Q})$ is contained in the class of cyclic groups of order coprime to $p$. It remains to show the reverse inclusion. Let $C_m$ be a cyclic group of order coprime to $p$, we will exhibit an element in $\Im(m_{\mathbb{F}_p , ss})$ which is in the kernel of every proper restriction map. Let $\mu$ denote the M\"obius function. One may check that the element $m_{(\mathbb{F}_p,ss)}(\sum_{n\mid m}\mu (n) n [C_n])$, is a non-zero element of $\Im(m_{(\mathbb{F}_p,ss)})(C_m)$ which restricts to zero on every proper subgroup. It follows that $C_m$ is coprimordial. 
	\end{proof}

\begin{corollary}\label{cor:injcyclic}
	For cyclic groups $C$ of order coprime to $p$ the map $m_{\mathbb{F}_p, ss}(C)$ is injective.
\end{corollary}
\begin{remark}
	Alternatively, one may show that the previous Corollary and preceding lemma follow immediately from Theorem \ref{thm:Artin} and the observation that the rank of $\Im(m_{\mathbb{F}_p,ss})$ is precisely the number of conjugacy classes of subgroups which are 
	 cyclic of order coprime to $p$.
\end{remark}
\begin{remark} \label{rem:inc} Corollary \ref{cor:injcyclic} combined with the analogous statement for $m_{\mathbb{Q}}$ (see \cite{BD}, \cite{Benson}) shows that $\ker(m_\mathbb{Q})(G)\subset \ker(m_{\mathbb{F}_p,ss})(G)$  for all $G$ with equality if $p\nmid |G|$. Extensive use will be made of this fact. Elements of $\ker(m_\mathbb{Q})$ will be called Brauer relations over $\mathbb{Q}$.
\end{remark}
\begin{example}
	The inclusion in Remark \ref{rem:inc} is in general strict. Over fields of characteristic $0$ cyclic groups admit no Brauer relations, but the kernel of $m_{\mathbb{F}_p,ss}$ need not be trivial. For example, there is a relation  $2[C_2]-[{e}]\in \K_{\F_2,ss}(C_2)$ for $C_2$ over $G_0(\mathbb{F}_2[C_2])$. Indeed, the regular representation of $C_2$ is indecomposable as an $\mathbb{F}_2[C_2]$-module, and has as its composition factors two copies of the trivial representation.
\end{example}
\begin{lemma}\label{lem:Cprel}
	Let $p$ be a prime, and let $C_p$ be the cyclic group of order $p$ then $\K_{\F_p,ss}(C_p)$ is generated by the relation $p[C_p]-[\{e\}]$.
\end{lemma}
\begin{proof}
	It is easy to verify that the claimed element of $b(C_p)$ is in $\K_{\F_p,ss}(C_p)$. Furthermore, since $\Im(m_{\mathbb{F}_p,ss})\supseteq \langle 1 \rangle_\mathbb{Z}$, the kernel has rank $1$. Clearly no integral relation divides $p[C_p]-[\{e\}]$ this completes the proof. 	
\end{proof}
Having identified the coprimordial groups for $\Im(m_{\mathbb{F}_p,ss })$,  Proposition \ref{prop:primordial} states that the primordial groups for $\Im(m_{\mathbb{F}_p,ss})$ are a subclass of $q$-quasi-elementary groups with cyclic part of order coprime to $p$. 
Note that it is possible to have $p=q$.
\begin{lemma}\label{lem:primordials}
	The primordial groups for $\Im(m_{\mathbb{F}_p,ss})$, are precisely the groups $H$ such that for some prime $q$, the minimal normal subgroup  $O^q(H)$ of $H$ whose associated quotient is a $q$-group, is cyclic of order prime to $p$.
\end{lemma}
\begin{proof}
	As previously stated it is a consequence of Lemma \ref{lem:copriss} and Proposition \ref{prop:primordial} that every primordial group is of this form.	In the case $p\nmid \#H$ the observation in Remark \ref{rem:inc} shows that $H$ is primordial for $\Im(m_{\mathbb{F}_p,ss})$ if and only if it is for $\Im(m_{\mathbb{Q}})$.  Theorem \ref{thm:sol} and the main result of \cite{Solomon} show that quasi-elementary groups are primordial for $\Im(m_{\mathbb{Q}})$.
	\\
	The remaining case to consider it when $H$ is $p$-quasi-elementary. If $H$ were not primordial for $\Im(m_{\mathbb{F}_p,ss})$, then in particular $1_{G_0(\mathbb{F}_p[H])}=1_{\Im(m_{\mathbb{F}_p,ss})(H)}\in \cI_{\Im(m_{\mathbb{F}_p,ss})}(H)$. It follows that for non-primordial groups there is a non-zero element $[H]+\sum_{K<H}[K]$ in $\K_{\F_p,ss}(H)$.   Consider the following two cases, the first where $H$ is not a $p$-group and the  second case where $H$ is a $p$ group.
	\begin{enumerate}
		\item If $H$ has a non-trivial coprime to $p$ cyclic subgroup $C$ then restriction of any element of $\K_{\F_p,ss}(H)$ to this subgroup must vanish by Corollary \ref{cor:injcyclic}. Since $p\mid [H:C]$ it follows from direct calculation that any relation must have coefficient of $1_{G_0(\mathbb{F}_p[H])}$ divisible by $p$. Since a finite group is primordial for a Green functor $\cF$ if and only if  $1_{\cF}$ is not in the image of proper inductions it follows that in this case $H$ must be primordial.
		\item Otherwise $H$ is a $p$-group and upon restriction to a central cyclic subgroup of order $p$ any relation must be of the form $a(p[C_p]-[e])$ by Lemma \ref{lem:Cprel}. Direct calculation shows that for any $K$ such that $C_p<K<H$ the restriction of $[K]$ to $b(C_p)$ is $[K:C_p][C_p]$. By an identical argument to the previous case, it follows the coefficient of $1_{G_0(\mathbb{F}_p[H])}$ is divisible by $p$. This completes the proof. 
	\end{enumerate}

\end{proof}
Thus, the primordial groups for $\Im(m_{\mathbb{F}_p , ss})$ are the set of quasi-elementary groups, for which the cyclic part $C$ is coprime to $p$.
\par
This may be phrased as an induction theorem.
\begin{corollary}\label{cor:ssind}
	Let $G$ be a finite group and let $p$ be a prime. Let $T$ be the set of conjugacy classes of primordial subgroups of $G$ for $\Im(m_{\mathbb{F}_p , ss})$, that is subgroups which are quasi-elementary with cyclic part coprime to $p$. Then:
	$$1_{G_0(\mathbb{F}_p[G])}=\sum_{H\in T} a_H \Ind_{G/H}(1_{G_0(\mathbb{F}_p[H])})$$
	where $a_H$ are integers.
\end{corollary}						
\section{Classification of $\Prim_{\K_{\F_p,ss}}(G)$ for Soluble $G$} \label{sec:ssclass}
In this section we give a necessary condition on a finite group $G$ for $\Prim_{\K_{\F_p,ss}}(G)$ to be non-trivial, and explicitly write down all such groups in the soluble case. The main results are summarised in the following theorem:
\par
\begin{theorem}\label{thm:main} Let $p$ be a prime, and $G$ be a finite group of order divisible by $p$. All Brauer relations for $G_0(\mathbb{F}_p[G])$ are linear combinations of relations induced and inflated from subquotients of the following forms:
	\begin{enumerate}
		\item a cyclic group $C_p$ of order $p$,
		\item non-cyclic $q$-quasi-elementary groups with order coprime to $p$,
		\item $(C_l\rtimes C_{q^r})\times (C_l\rtimes C_{q^s})$ for primes $q,l$ with $l\ne p$ and the action faithful,
		\item an extension $1 \rar S^d \rar E \rar H \rar 1$ with $S$ simple, $d$ a positive integer, $H=C\rtimes Q$ a quasi-elementary group whose cyclic part $C$ is of order coprime to $p$, and $S^d$ is a unique minimal normal subgroup of $E$.
	\end{enumerate}

\end{theorem}
\par
 The first step is to give a description of $\Prim_{\K_{\F_p,ss}}(G)$ based on the proper quotients of $G$ for any non-primordial group $G$.
\begin{lemma}\label{lem:SES}
	Let $G$ be a finite group which admits primitive relations over $\mathbb{F}_{p},ss$ then $G$ is an extension of the following form:
	\begin{align}\label{eqn:SESss}
	1\rar S^d \rar G \rar H \rar 1 
	\end{align}
	where $S$ is a finite simple group, $d\ge 1$, $H$ is  quasi-elementary with cyclic part coprime to $p$. 
	Furthermore, if $G$ is not primordial for $\Im(m_{\mathbb{F}_p,ss})$ then $\Prim_{\K_{\F_p,ss}}(G)$ is as follows:
	\begin{enumerate}
		\item if all quotients of $G$ are cyclic of order coprime to $p$, then $\Prim_{\K_{\F_p,ss}}(G)$ is isomorphic to $\mathbb{Z}$, 
		\item if all quotients of $G$ are $q$-quasi-elementary with cyclic part coprime to $p$ and at least one of them is not cyclic of order coprime to $p$, then $\Prim_{\K_{\F_p,ss}}(G)=\mathbb{Z}/{q\mathbb{Z}}$,
		\item otherwise it is trivial. 
	\end{enumerate}
	Furthermore, in all cases $Prim(G)$ is generated by any Brauer relation of the form $[G]+\sum_{H<G}n_H [H]$.
\end{lemma}
\begin{proof}
	  The primordial groups for $\Im(m_{\mathbb{F}_p,ss})$  are identified in Lemma \ref{lem:primordials}, while Corollary \ref{cor:ssind} shows that there exists a Brauer relation of the form $[G]+\sum_{H<G}n_H [H]$ for non-primordial $G$. It then follows from \cite[Theorem 4.7]{GFI} that $G$ is an extension of the claimed form, and that, for such extensions which are non-primordial, any relation $[G]+\sum_{H<G}n_H [H]$ must generate $\Prim_{\K_{\F_p,ss}}(G)$.
	   \\
	If all proper quotients are cyclic of order coprime to $p$ then they are coprimordial for $\Im(m_{\mathbb{F}_p,ss})$ and so primordial for $\Im(m_{\mathbb{F}_p,ss})_\mathbb{Q}$ by Proposition \ref{prop:primordial}.  \cite[Corollary 4.10]{GFI} then shows that $\Prim_{\K_{\F_p,ss}}(G)=\mathbb{Z}$ in this case.\\
	It remains to consider the case where there exists a quotient which is $q$-quasi-elementary with cyclic part prime to $p$ but not cyclic of order prime to $p$. First note that for non-cyclic $q$-quasi-elementary groups \cite[Theorem 1]{Solomon} shows that there exists a relation over $\mathbb{Q}$ of the form $qG+\sum_{H\lneq G}a_H [H]$ with the $a_H$ in $\mathbb{Z}$ and hence by Remark \ref{rem:inc} over $G_0(\mathbb{F}_p[-])$. It follows from Lemma \ref{lem:primordials} that the coefficient of  $G$ in any relation is not $1$. Furthermore, inflating the relation in Lemma \ref{lem:Cprel}, shows that there is a relation $p[C_{p^r}]-[C_{p^{r-1}}]$ for any $C_{p^r}$.  Corollaries 4.11 and 4.12 of \cite[Theorem 4.9]{GFI} then give the claimed result.
\end{proof}
We now classify which groups $G$ of the form \eqref{eqn:SESss} have $\Prim_{\K_{\F_p,ss}}(G)$ non-trivial.
\begin{corollary}\label{cor:nonsol}
	Let $G$ be a finite insoluble group which admits a primitive relation over $\mathbb{F}_{p},ss$. Then $G$ is of the form described in Lemma \ref{lem:SES} with $S$ non-cyclic, $H$ injects into $\Out(S^d)$ and no proper non-trivial subgroup of $S^d$ is normal in $G$. Furthermore, every such group admits a primitive relation.
\end{corollary}
\begin{proof}
	The corollary follows immediately from Lemma \ref{lem:SES} upon noting that for  such an extension $G$ every quotient is quasi-elementary with cyclic part coprime to $p$ if and only if the action of $G$ on $S^d$ is faithful and no proper non-trivial subgroup of $S^d$ is normal in $G$. Since the centre of $S^d$ is trivial the action of $G$ is faithful if and only if $G/(S^d)=H\hookrightarrow \Out(S^d)$.
\end{proof}
Since the inclusion in Remark \ref{rem:inc} is an equality when we restrict to groups of order coprime to $p$ and since there is a full classification of Brauer relations in characteristic zero in \cite{BD} we now need only consider  $G$ whose order is divisible by $p$. Furthermore, in light of Corollary \ref{cor:nonsol} we restrict to the soluble case.
We will make repeated use of the following result.
\begin{lemma}\label{lem:SESsplit}
	Let $G$ be a finite group, and let $W$ an abelian normal subgroup
	with quotient $H$. Suppose that there exists a normal subgroup $K$ of $H$
	such that $\gcd(\#K,\#W)=1$ and such that no non-identity element
	of $W$ is fixed under the natural conjugation action of $K$ on $W$.
	Then $G\cong W\rtimes H$.
\end{lemma} 
\begin{proof}
	We may view $W$ as a module under $H$.
	Since $K$ and $W$ have coprime orders, the cohomology group
	$H^i(K,W)$ vanishes for $i> 0$, so the Hochschild--Serre spectral sequence \cite[Theorem 6.3]{Brown}
	gives an exact sequence
	$$
	H^2(H/K,W^K)\rightarrow H^2(H,W)\rightarrow H^2(K,W).
	$$
	The last term in this sequence also vanishes by the coprimality assumption,
	while the first term vanishes, since $W^K$ is assumed to be trivial.
	So $H^2(H,W)=0$, and so the extension $G$ of $H$ by $W$ splits.
\end{proof}
\begin{theorem}\label{thm:sol} Suppose that $G$ is a soluble group, of order divisible by $p$, and $\Prim_{\K_{\F_p,ss}}(G)$ is non-trivial. Then $G$ is of one of the following forms:
	\begin{enumerate}
		\item a $p$-group or,
		\item a $p$-quasi-elementary group or,
		\item $(C_l)^d\rtimes H$ with $l$ a prime, $H$ quasi-elementary, with cyclic part coprime to $p$, acting faithfully and irreducibly on $(C_l)^d$ or,
		\item  $(C_l\rtimes C_{p^r})\times (C_l\rtimes C_{p^s})$  with faithful action and $l$ a prime.
	\end{enumerate}
\end{theorem}

\begin{proof}
	The first two parts follow from taking trivial extensions of $p$-groups and quasi-elementary groups respectively in Lemma \ref{lem:SES}. Since by assumption $p$ must divide the order of $G$ trivial extensions of cyclic groups of coprime to $p$ are not included in the list. Note that for $G$ as in Lemma \ref{lem:SES} to be soluble, and not $p$-quasi-elementary, is equivalent to taking $S=C_l$. By Lemma \ref{lem:SES}, $G$ is therefore an extension of the form:
	$$1\rar W:=(C_l)^d \rar G \rar H:=C\rtimes Q \rar 1 ,$$
	where $H$ is quasi-elementary with cyclic part coprime to $p$. We wish to show that under our assumptions $G$ is a split extension or lies in case (1) or (2) of the theorem. If $l\nmid \#H$ then by Schur-Zassenhaus the result follows. Suppose that that $l\mid \#H$, we split into the following cases $C$ is trivial, $Q$ is trivial, and neither $C$ nor $Q$ is trivial. 
	\begin{enumerate}[(i)]
		\item \textbf{$C$ is trivial}.
		In this case $G$ is a $q$-group.
		\item \textbf{$Q$ is trivial}. Either $G$ is an $l$ group or $C$ admits a subgroup  $C_{l'}$ of order coprime to $l$. Either $W^{C_{l'}}$ is trivial, in which case Lemma \ref{lem:SESsplit} with $K=C_{l'}$ allows us to conclude $G$ is split or $\{e\}\ne V:=W^{C_{l'}} \triangleleft G$. Now $G/V$ must be quasi-elementary and $(W/V)^{C_{l'}}=\{e\}$. Since this quotient must be $l$-quasi-elementary, in fact $W=V$ and  $G$ is therefore, $l$-quasi-elementary with cyclic part coprime to $p$.  We must have $l=p$ as $p\mid \#G$ by assumption.
		\item \textbf{Both $Q$ and $C$ are non trivial}. If $C$ is an $l$ group then let $L$ denote the $l$-sylow subgroup of $G$, clearly $L\triangleleft G$. Let $\Phi(L)$ be the Frattini subgroup of $L$, if it is trivial then $L=C_l^n$ and $G=C_l^n\rtimes Q$ and so $G$ is a split extension as claimed.
		Otherwise $\Phi(L)\triangleleft G$ and $G/\Phi(L)$ must be $q$-quasi-elementary with cyclic part prime to $p$ as $\Prim_{\K_{\F_p,ss}}(G)$ is non-trivial. Thus $L/\Phi(L)$ must be cyclic and we see that $l\ne p$. As $L/\Phi(L)$ is cyclic,  $L$ is also cyclic. It follows that in this case that $G$ is $p$-quasi-elementary.
		\par 
		If $C$ is not an $l$-group, then let $K=C_{l'}$ as before, if $W^K$ is trivial then the extension is split by Lemma \ref{lem:SESsplit}, if not then let $V=W^K$, the quotient $G/V$ must be $q$-quasi-elementary with cyclic part coprime to $p$. \\
		If $l=p$, this forces $q=p$ and $W^K=W$ so $G$ is $p$-quasi-elementary. 
		\\
		In the remaining case $l\ne p$, the quotient $G/V$ must be quasi-elementary with cyclic part prime to $p$. The image of $K$ in the quotient must act fixed point freely on $W/V$ that is $(W/V)^K$ is trivial. Thus either $W/V$ is non trivial and $G/V$ is $l$-quasi-elementary of order coprime to $p$, and thus so is $G$ a contradiction, or $V=W$ and $G/K$, and hence $G$, must be $p$-quasi-elementary.
	\end{enumerate}
	It remains to consider the split sequence $$1\rar (C_l)^d \rar G \rar H \rar 1 .$$
	\par 
	We subdivide into three cases; in the first $l=p$, and $H$ is $p$-quasi-elementary with $q\ne p$, in the second $l\ne p$ and $H$ is $p$-quasi-elementary, and finally  $l=p$ and $H$ is $p$-quasi-elementary. In all cases the cyclic part of $H$ is coprime in order to $p$.
	\begin{enumerate}[(a)]
		\item {\bf In the first case} $G=(C_p)^d\rtimes(C\rtimes Q)$ where $Q$ is a $q$-group and $p\nmid\#C$. 
		{\bf \emph{ Faithfulness:}} Suppose 
		$C\rtimes Q$ acts with kernel $K\ne\{e\}$.  Then the quotient $G/K$ must be $q$-quasi-elementary with cyclic part prime to $p$, but $p\ne q$ so the image of $W=(C_p)^d$ must lay in the cyclic part of the quotient, a contradiction.
		{\bf \emph{Irreducibility:}} Suppose that the action were reducible so there exists $V=(C_p)^{d_1}\triangleleft G$ with $d_1<d$. The corresponding quotient must be $q$-quasi-elementary of order coprime to $p$ a contradiction. We conclude that the action is faithful and irreducible.
		\item {\bf In the second case} $G=(C_l)^d\rtimes(C\rtimes P)$ where $l\ne p$ and $P$ is a $p$-group. {\bf \emph{Faithfulness:}} Suppose that $C\rtimes P$ acts with kernel  $K\ne\{e\}$ then the quotient $G/K$ must be $p$-quasi-elementary. In particular, this forces $d=1$ and  since the image of $C$ and $W$ must both be in the cyclic part of the quotient we have $K\ge C$ so $G$ was $p$-quasi-elementary. {\bf \emph{ Irreducibility:}} Assuming that the action is faithful, either it is irreducible in which case we find ourselves in case 3 of the theorem or it is reducible. If the action on  $W:=C_l^d$  were reducible, then there exists $V <W$ a normal subgroup of $G$, the quotient group $G/V$ must then be $p$-quasi-elementary. By assumption $l\ne p$ so the $l$-Sylow of $G/V$ must be cyclic, if  $G=(C_l)^d\rtimes(C\rtimes P)$  with $l\mid \#C$  then this would be impossible. We conclude that if the action is reducible then  $G=((C_l)\times V)\rtimes(C\rtimes P)$  with $l \nmid \#(C\rtimes P)$ with semisimple action, quotienting by $C_l$ shows via an identical argument that $V\cong C_l$ and so we are in part (4) of the theorem. 
		\item {\bf Finally}  $G=(C_p)^d\rtimes(C\rtimes P)$. {\bf \emph{ Faithfulness:}} We claim either this group is $p$-quasi-elementary or the action is faithful. If the action had a kernel $K$ the quotient by the kernel must be $p$-quasi-elementary, and so $C\le K$ and $G$ was already $p$-quasi-elementary. {\bf \emph{Irreducibility:}} If the action is faithful then we claim that it must be irreducible. Assume otherwise, then there exists  $V=(C_p)^{d_1}\triangleleft G$ such that $G/V$ is $p$-quasi-elementary. This forces $C$ to act trivially on $(C_p)^d/V$ and so $({(C_p)^d})^C\ne 0$ as $C$ has order coprime to $p$ and thus the action is semisimple. By assumption $G/{({(C_p)^d})^C}$ is $p$-quasi-elementary, now we may assume that the complement of ${({(C_p)^d})^C}$ in $(C_p)^d$ is non-trivial (else $C$ is normal and $G$ quasi-elementary) and thus the quotient is not $p$-quasi-elementary. Thus either the action is faithful and irreducible or $G$ is quasi-elementary.
	\end{enumerate}
\end{proof}
\begin{theorem}\label{thm:noprimss}
	Let $G$ be a $p$-quasi-elementary group which is not cyclic of order $p$, then $\Prim_{\K_{\F_p,ss}}(G)$ is trivial.
\end{theorem}
\begin{proof}
	The rank of the space of relations of $G$ is the number of conjugacy classes of subgroups of $G$ which are not cyclic of order coprime to $p$. We will construct a sublattice of imprimitive relations which has full rank, then proceed to show it is saturated.
	\\
	Let $G=C_m\rtimes P$ with $P$ a fixed Sylow $p$-subgroup of $G$ and $p\nmid m$. Subgroups of $G$ which are not cyclic of order coprime to $p$ are determined up to conjugacy by their intersection with $C_m$, which is characteristic and hence normal in $G$, and their intersection with $P$ which is non-trivial by assumption.
	We fix a labelling on the subgroup lattice of $P$ up to $G$-conjugacy let $P_{i,j}$ be the $j$th subgroup of size $p^i$. Subgroups which are not cyclic of order coprime to $p$ are then characterised up to conjugacy as $C_s\rtimes P_{i,j}$ where $s\mid m$ and $i\ne0$. Each such subgroup admits an imprimitive relation inflated from any quotient $C_p$ namely $p[C_s\rtimes P_{i,j}]-[C_s \rtimes P_{i-1,k}]$ where $ P_{i,j} > P_{i-1,k}$. Note that as every maximal subgroup of a $p$-group has index $p$ we may use these relations to create the relation  $[C_s\rtimes P_{i,j}]-[C_s \rtimes P_{i,k}]$ for any $P_{i,k}$. \\
	{\bf \emph{The sublattice:}}  We now exhibit a full rank sublattice of imprimitive relations. We form the span of  $p[C_s\rtimes P_{i,0}]-[C_s \rtimes P_{i-1,0}]$ as we range over $s\mid m$ and $i>1$ along with relations $[C_s\rtimes P_{i,0}]-[C_s \rtimes P_{i,j}]$ for $s\mid m,$ and  $i,j>0$. Clearly this set is linearly independent and of the correct size, so we have a full rank sublattice;
	\begin{align*}
	\mathcal{L}=\langle p[C_s\rtimes P_{i,0}]-[C_s \rtimes P_{i-1,0}] , [C_s\rtimes P_{i,0}]-[C_s \rtimes P_{i,j}] |i\in I, j \in J\rangle_\mathbb{Z}.
	\end{align*}
	{\bf \emph{Saturation:}} The sublattice $\mathcal{L}$ is in fact, saturated, suppose that there exists a relation $\theta$ such that $n\theta=\sum_{s \mid m} (\sum_{i>0} a_{s,i}( p[C_s\rtimes P_{i,0}]-[C_s \rtimes P_{i-1,0}])+\sum_{j>0}b_{s,i,j}([C_s\rtimes P_{i,0}]-[C_s \rtimes P_{i,j}]))$ for a relation $\theta$ we seek to show that $n\mid a_{s,i},b_{s,i,j}$ for all $s,i,j$ in the indexing sets. Since the coefficient of $[C_s \rtimes P_{i,j}]$ on the right hand side is $b_{s,i,j}$ we may conclude that $n\mid b_{s,i,j}$ , subtracting all terms in the second sum from both sides we then have $n\theta'=\sum_{s \mid m}\sum_{i>0} a_{s,i}( p[C_s\rtimes P_{i,0}]-[C_s \rtimes P_{i-1,0}])$, where the coefficient of $[C_s \rtimes P_{0,0}]$ on the right hand side
	is $-a_{s,1}$ so that $n\mid a_{s,1}$. Now the coefficient of $[C_s \rtimes P_{i,0}]$ is $pa_{s,i}-a_{s,i+1}$ and so if $n\mid a_{s,i}$ then $n\mid a_{s,i+1}$ and by induction $n\mid a_{s,i}$ for all $s,i$ in the indexing set.\\
	Thus we have a full rank saturated sublattice  of imprimitive relations, it follows that every relation is imprimitive.
\end{proof}

Theorems \ref{thm:sol} and \ref{thm:noprimss} allow us to conclude the following.
\begin{theorem}\label{thm:finalsol}
	Let $G$ be a finite soluble group, $p$ a prime, $\Prim_{\K_{\F_p,ss}}(G)$ is non-trivial if and only if:
	\begin{enumerate}
		\item  $G\cong C_p$, then $\Prim_{\K_{\F_p,ss}}(G)\cong \mathbb{Z}$ or,
		\item $G\cong (C_l)^d\rtimes H$ with $H$ $q$-quasi-elementary with cyclic part coprime to $p$ acting faithfully and irreducibly on $(C_l)^d$, then $\Prim_{\K_{\F_p,ss}}(G)\cong \mathbb{Z}/{q\mathbb{Z}}$ or $\mathbb{Z}$ if $H$ is cyclic of order coprime to $p$ or,
		\item for $l\ne p$ a prime $G=(C_l\rtimes C_{q^r})\times (C_l\rtimes C_{q^s})$, $\Prim_{\K_{\F_p,ss}}(G)\cong \mathbb{Z}/{q\mathbb{Z}}$ or,
		\item$G$ is quasi-elementary of order coprime to $p$ with $\Prim_{\K_{\F_p,ss}}(G)$ as over $\mathbb{Q}$.
	\end{enumerate} 
\end{theorem}
Combining this with Corollary \ref{cor:nonsol} gives Theorem \ref{thm:main}.
\section{Some Explicit Relations}
We now establish generators of $\Prim_{\K_{\F_p,ss}}(G)$ for the soluble groups of order divisible by $p$
admitting primitive relations. Since the inclusion in Remark \ref{rem:inc} becomes an equality when $p\nmid \#G$, this combined with the classification in \cite[Theorem A]{BD} completely determines all Brauer relations for soluble groups over $G_0(\mathbb{F}_p[G])$. 
\begin{lemma}
	The group $\Prim_{\K_{\F_p,ss}}(C_p)$ is generated by $p[C_p]-[\{e\}]$. The group $\Prim_{\K_{\F_p,ss}}((C_l)^d\rtimes H)$ with $H$ quasi-elementary acting faithfully is generated by the same relation as over $\mathbb{Q}$  for $d>1$.
\end{lemma}
\begin{proof}
	The first claim is identical to Lemma \ref{lem:Cprel}. For the second, note that the Brauer relation over $\mathbb{Q}$ for $d\ge 2$ given in \cite[Proposition 6.4]{BD}  has coefficient of $[G]$ equal to $1$ and, as its still a relation in this setting (see Remark \ref{rem:inc}), it must generate $\Prim_{\K_{\F_p,ss}}(G)$ by Lemma \ref{lem:SES}, this gives the second statement.
\end{proof}
Note that a quasi-elementary group acting faithfully on a cyclic group of prime order must be cyclic so the only remaining case (corresponding to $d=1$) is the case of a coprime to $p$ cyclic group acting faithfully on a cyclic group of order $p$.
\begin{lemma}\label{lem:1dim}
	Let $G=C_p\rtimes C_{mq^r}$ with faithful action then $\Prim_{\K_{\F_p,ss}}(G)$ is generated by the same relation as presented in \cite[Proposition 6.5]{BD} over $\mathbb{Q}$ unless $ m=1$ in which case it is generated by the following relation $-[C_{q^r}]+(p-1)/q^r[C_p]+[C_p\rtimes C_q^r]$.
\end{lemma}
\begin{proof}
	We explicitly construct such a relation and since the coefficient of $[G]$ is $1$ it must generate $\Prim_{\K_{\F_p,ss}}(G)$. In the case $m$ is non-trivial we simply apply Remark \ref{rem:inc} and use the Brauer relation over  $\mathbb{Q}$ in \cite[Proposition 6.5]{BD}. Otherwise $G= C_p\rtimes C_{q^r}$, and we have the following relation: $$p[C_p]-[e],$$ induced from the subgroup $C_p$, and the primitive relation over $\mathbb{Q}$, $$[C_{q^{r-1}}]-q[C_{q^r}]-[C_p\rtimes C_{q^{r-1}}]+q[C_p\rtimes C_{q^r}].$$ 
	Using a linear combination of the two identified relations we can produce a third which is a multiple of the relation in the statement. We proceed by induction on $r$. 
	\par 
	
	If $r=1$ then taking a linear combination: $$ \alpha ([\{e\}]-q[C_{q}]-[C_p]+q[C_p\rtimes C_{q}])-\beta ((p[C_p]-[e])),$$ 
	and setting $\alpha=1, \beta=-1$ gives  $q$ times the relation: $$-[C_q]+(p-1)/q[C_p]+[C_p\rtimes C_q].$$\\
	If $r> 1$, assume that for $s<r$ the group $C_p\rtimes C_{mq^s}$ admits the relation: $$-[C_{q^s}]+(p-1)/q^s[C_p]+[C_p\rtimes C_q^s].$$ Then the $\mathbb{Q}$ relation for $G$; $$[C_{q^{r-1}}]-q[C_{q^r}]-[C_p\rtimes C_{q^{r-1}}]+q[C_p\rtimes C_{q^r}],$$ plus  $$-[C_{q^{r-1}}]+(p-1)/q^{r-1}[C_p]+[C_p\rtimes C_q^{r-1}],$$ the induced relation from $C_p \rtimes C_{q^{r-1}}$, gives $q$ times the claimed relation. 
\end{proof}

\newpage
\bibliographystyle{abbrv}

\bibliography{newbib}

\end{document}